\documentclass[a4paper,12pt]{amsart}

\usepackage{amsmath,amsthm,amscd,mathrsfs }

\usepackage{amssymb}
\usepackage[all,cmtip]{xy}

\newtheorem{Theorem}{Theorem}[section]
\newtheorem{Lemma}[Theorem]{Lemma}
\newtheorem{Corollary}[Theorem]{Corollary}
\newtheorem{Proposition}[Theorem]{Proposition}
\newtheorem{Remark}[Theorem]{Remark}
\newtheorem{Notation}[Theorem]{Notation}

\newtheorem{Example}[Theorem]{Example}
  
\newtheorem{Definition}[Theorem]{Definition}

\newcommand{\Hom}{\mathrm{Hom}}

\DeclareMathOperator{\Ext}{Ext}
   \DeclareMathOperator{\Assh}{Assh}
  
  \DeclareMathOperator{\height}{height}
\DeclareMathOperator{\depth}{depth}  
  \DeclareMathOperator{\Coker}{Coker}
 \DeclareMathOperator{\Ass}{Ass} \DeclareMathOperator{\Supp}{Supp} 
\DeclareMathOperator{\Min}{Min}   \DeclareMathOperator{\Spec}{Spec}
       
     \DeclareMathOperator{\Ho}{H}

      \DeclareMathOperator{\Anc}{Anc}

\usepackage{amsmath, amsthm}
\usepackage{chemarr}

\begin{document}

\title
{Canonical modules of complexes}

\author{Maryam Akhavin}
\author{Eero Hyry}

\address{
Mathematics and Statistics\\
School of Information Sciences\\ 
University of Tampere\\
FIN-33014 Tampereen yliopisto\\ 
Finland}

\email{maryam.akhavin@uta.fi}
\email{eero.hyry@uta.fi}

\begin{abstract}
We define the notion of the canonical module of a complex. We then consider Serre's conditions
for a complex and study their relationship to the local cohomology of the canonical module and its
ring of endomorphisms.
\end{abstract}

\date{}
\maketitle

\section{Introduction}

The notion of the canonical module of a ring is an important tool in commutative algebra. P.~Schenzel defined in~\cite{schenzellec} the notion of the canonical 
module of a module. The canonical module always satisfies Serre's condition $(S_2)$. Schenzel related 
higher Serre's conditions to the vanishing of certain local cohomology modules of the canonical module. The purpose of this article is to extend these results to 
complexes. We utilize the powerful tools of hyperhomological algebra.

Let $(R,m) $ be a Noetherian local ring admitting a dualizing complex. We work within the derived category $D^f_b(R)$ of bounded complexes of $R$-modules with finitely generated homology. 
Generalizing the work of Schenzel, we define for any complex $M\in D^f_b(R)$ and any $i\in \mathbb Z$ the $i$-th module of deficiency $K^i_M$ by setting
$K^i_M=\Ho_i({\mathbf R}\Hom_R(M,D_R)$, where $D_R$ denotes the dualizing complex of $R$ normalized so that $\Ho_{\dim R}(D_R)\not=0$  and $\Ho_i(D_R)=0$ for $i>\dim R$. The canonical module of $M$ is then $K_M=K^{\dim M}_M$. Note that by local duality the local cohomology module $\Ho_m^i(M)$ is the Matlis dual of $K^i_M$. In particular, modules of deficiency measure how far the complex is from being Cohen-Macaulay.

Given $k\in \mathbb N$, we say that a complex $M$ satisfies Serre's condition $(S_k)$ if $$\depth_{R_p}M_p\ge \min(k-\inf M_p, \dim M_p)$$ for all prime ideals $p\in \Supp M$. It 
is convenient to consider complexes satisfying the condition $\dim_RM=\dim_{R_p}M_p+\dim R\slash p$ for every $p\in\Supp_{R}M$. Here $\Supp M$ means the homological support of $M$. It then follows from our Theorem~\ref{maintheorem} that $(S_k)$ is equivalent to the natural homomorphism $$\Ext^{-i}_R(M,M)\rightarrow  K^{i+\dim M}_{M\otimes_R^L K_M}$$ being bijective for all  $i\geq -k+2$, and injective for $i=-k+1$. Note that $K_{M\otimes_R^L K_M}=\Hom_R(K_M,K_M)$. It makes also sense, for any $l\in \mathbb Z$, to look at the condition $(S_{k,l})$ saying that $$\depth_{R_p}M_p\ge \min(k-l, \dim M_p)$$ for all prime ideals $p\in \Supp M$. Observe that $(S_k)$ always implies $(S_{k,\sup M})$. It now turns out in Corollary~\ref{schenzeltheorem} that $(S_{k,l})$ is equivalent to the natural homomorphism $\Ho_i(M) \rightarrow K^{i+t}_{K_M}$ being bijective for $i\geq l-k+2$, and injective for $i =l-k+1$. In the case $M$ is a module and $l=0$, this reduces to the result of Schenzel mentioned in the beginning.

Finally, we look at the complex $M^\dagger:={\mathbf R}\Hom_R(M,D_R)$. Suppose that $\sup M=\sup M_p$ for all $p\in \Supp M$. If $M^\dagger$ satisfies Serre's condition $(S_2)$, it comes out in 
Corollary~\ref{H_SK_M2''} that $K_M\cong K_{\Ho_s(M)}$, where $s=\sup M$. Combining this with the observation of Lipman, Nayak and Sastry in~\cite[Proposition 9.3.5]{lnsfunctorial} that the Cousin complex of a complex depends only on the
top homology of $M^\dagger$ i.e.~on the canonical module, we can relate the Cousin complex of the complex $M$ to that of the module $\Ho_s(M)$. More precisely, we show in Proposition~\ref{MK_ME2} that in the above situation $$E_{{\mathcal D}(M)}(M)\cong \textstyle\sum^{s}E_{{\mathcal D}(\Ho_s(M))}(\Ho_s(M)),$$ where ${\mathcal D}(M)$ and ${\mathcal D}(\Ho_s(M))$ denote the dimension filtrations of $M$ and $\Ho_s(M)$
respectively.
 
\section{Preliminaries}

The purpose of this section is to fix notation and recall some definitions and results of hyperhomological algebra relevant to this article. As a general reference, we mention~\cite{C.L.N} and references therein. For more details, see also~\cite{sequence}, \cite{sequence2}, \cite{f.vanishloho} and~\cite{hart}. 

In the following $R$ is always a commutative Noetherian ring. If $R$ is local, then $m$ denotes the maximal ideal and
$k$ the residue field of $R$.

Throughout this article we work within the derived category $D(R)$ of $R$-modules. We use homological grading so that the objects of $D(R)$ are complexes of $R$-modules
of the form $$M: \quad\ldots\stackrel{d_{i-1}}{\rightarrow} M_{i+1}\stackrel{d_{i+1}}{\rightarrow} M_i\stackrel{d_i}{\rightarrow} M_{i-1} \ldots \quad .$$
The derived category is triangulated, the suspension functor $\Sigma$ being defined by the formulas $(\Sigma M)_n= M_{n-1}$ and $d^{\Sigma M}_n=-d_{n-1}$. 
The symbol ``$\simeq$'' is reserved for isomorphisms in $D(R)$. We use the subscripts ``$b$'', ``$+$'' and ``$-$'' to denote the homological boundness, the homological 
boundness from below and the homological boundness from above, respectively. The superscript ``$f$'' denotes the homological finiteness. So the full subcategory of $D(R)$ consisting of complexes with finitely generated homology modules is denoted by $D^f(R)$. As usual, we identify the 
category of $R$-modules as the full subcategory of $D(R)$ of complexes $M$ satisfying $\Ho_i(M)=0$ for $i\neq 0$. For a complex
$M\in D(R)$, by $\sup M$ and $\inf M$, we mean its homological supremum and infimum.  Let $M$ and $N$ be complexes of $R$-modules. We use 
the standard notations $M\otimes_R^LN$ and ${\bf R}\Hom_R(M,N)$ for the left derived tensor product complex 
and the right derived homomomorphism complex, respectively. Moreover, we set $\Ext_R^{-i}(M,N)=\Ho_i({\bf R}\Hom_R(M,N))$ for all $i\in \mathbb Z$.

The \textit{support} of a complex $M\in D(R)$ is the set $$\Supp_RM=\left\{p\in \Spec R\mid M_p\not\simeq 0 \right\}.$$ The \textit{Krull dimension} $$ \dim_RM  =\sup \left\{\dim R/p -\inf M_p 
\mid p\in {\Supp_RM}\right\}.$$  If $M\in D_+(R)$, then $\dim M\ge -\inf M$. For every $p\in \Spec R$, there is an inequality
\begin{equation}\label{dimineq}
\dim_R M \ge \dim_{R_p} M_p + \dim R/p.
\end{equation}
Also note the formula
\begin{equation}\label{dim} 
\dim_RM=\sup \bigl\{\dim_R\Ho_i(M)-i\mid i\in \mathbb Z\bigr\} 
\end{equation} 
(see~\cite[Lemma 6.3.5]{918}).

Let $M\in D_+(R)$. We recall from~\cite[Definitions 2.1]{sequence2} that a prime ideal $p\in \Supp_RM$ is called an \textit{anchor prime} for $M$, if $\dim_{R_p}M_p=-\inf M_p$. The set of all anchor primes for $M$ is denoted by $\Anc_RM$. The anchor primes play the role of minimal primes for complexes. Note that $p\in \Anc_R(M)$, if $\dim R\slash p=\dim_RM + \inf M_p$.  

\begin{Definition}Let $(R,m)$ be a ring, and let $M\in D_+(R)$. We say that  $M$ is \textit{equidimensional}, if 
$$\Anc_R(M)=\{p\in \Supp M\mid \dim R\slash p=\dim_RM + \inf M_p\}.$$
\end{Definition}

\begin{Proposition}\label{2222}Let $(R,m)$ be a catenary ring, and let $M\in D_+(R)$. Then the following conditions are equivalent:
\begin{itemize}
\item[a)] $M$ is equidimensional;
\item[b)] $\dim_RM=\dim_{R_p}M_p+\dim R\slash p$ for every $p\in\Supp_{R}M$.
\end{itemize}
\end{Proposition}
\begin{proof}
\smallskip
\noindent $a)\Rightarrow b):$ Let $p\in\Supp_RM$. By inequality $(\ref{dimineq})$ we have $$\dim_{R_p}M_p\geq\dim_{R_q}M_q+\dim R_p\slash {qR_p}$$ for every $qR_p\in\Supp_{R_p}M_p$. Take now $q\in\Min\Supp_RM$ such that $q\subset p$. Then
\begin{align*} 
\dim_{R_p}M_p+\dim R\slash p&\geq\dim_{R_q}M_q+\dim R_p\slash {qR_p}+\dim R\slash p\\&=\dim_{R_q}M_q+\dim R\slash q\\&=-\inf M_q+\dim R\slash q\\&=\dim_RM.
\end{align*} 
Here the first equality holds true, since $R\slash q$ is a catenary integral domain. The second equality comes from the fact that $\Min\Supp_RM\subset\Anc_RM$ (see~\cite[Theorem 2.3 (a)]{sequence2}).
The last equality then follows from the equidimensionality of $M$. Since the converse inequality comes from inequality $(\ref{dimineq})$, we are done.
\smallskip

\noindent $b)\Rightarrow a):$ This is clear, since now 
$$-\inf M_p=\dim_{R_p}M_p=\dim_RM-\dim R\slash p$$
for every $p\in\Anc_R M$
\end{proof}

If $(R,m)$ is local and $M\in D_-(R)$, then the \textit{depth} of $M$ is defined by the formula $\depth_RM =-\sup {\bf R}\Hom_R(k,M)$. One has $\depth_RM\geq-\sup M$. Moreover, when $M\not\simeq 0$, the equality holds if and only if $m\in \Ass_R\Ho_{\sup M}(M)$ (see~\cite[Observation 5.2.5]{918}). If $\depth M=\dim M$, then $M$ is said to be \textit{Cohen-Macaulay}. For any ring $R$ and $M\in D_-(R)$, a prime ideal $p\in \Supp_RM$ is called an \textit{associated prime} ideal of $M$ if $\depth_{R_p}M_p=-\sup M_p$. The set of all associated primes of $M$ is denoted by $\Ass_RM$.
Furthermore, when $M\not\simeq 0$,  
 \begin{equation}\label{haft}
  p\in\Ass_R\Ho_{\sup M}(M)
 \ 
 \hbox{if and only if}
 \ 
 \depth_{R_p}M_p=-\sup M
 \end{equation} 
by~\cite[A.6.1.2]{C.L.N}.

If $(R,m)$ is a local ring, the \textit{derived local cohomology} functor with respect to  $m$ is denoted by ${\bf R}\Gamma_m$. As usual, we set $$\Ho_m^i(-)=\Ho_{-i}({\bf R}\Gamma_m(-))$$ for all $i\in \mathbb Z$. Note that
\begin{equation}\label{lidagger}-\inf {\bf R}\Gamma_m(M)=\dim_RM\end{equation}
and
\begin{equation}\label{lsdagger}-\sup{\bf R}\Gamma_m(M)=\depth_RM \end{equation}
(see~\cite[2.4]{f.vanishloho}).
If $R$ admits a dualizing complex, we denote by $D_R$ the dualizing complex normalized with $\sup D_R=\dim R$ and $\inf D_R=\depth R$. The \textit{dagger dual} of a complex $M\in D^f_b(R)$ is 
$$M^\dagger = {\bf R}\Hom_R(M,D_R).$$ We obtain a contravariant functor $(-)^\dagger\colon D^f_b(R)\rightarrow D^f_b(R)$.  The canonical morphism $M \rightarrow M^{\dagger\dagger}$ induces the biduality $M\simeq M^{\dagger\dagger}$, which is called the \textit{dagger duality} for $M$. The \textit{local duality} says that 
\begin{equation}\label{LDT}{\bf R}\Gamma_m(M)\simeq\Hom_R(M^\dagger,E_R(k)),\end{equation} where $E_R(k)$ denotes the injective envelope of $k$. We will frequently use the formulas
\begin{equation}\label{sdagger}\sup M^{\dagger}=\dim_RM\end{equation}
and
\begin{equation}\label{idagger}\inf M^{\dagger}=\depth M\end{equation} 
(see~\cite[Proposition 7.2.7]{918}).
By biduality, we then have $\sup M=\dim_R M^\dagger$ and $\inf M=\depth_R M^\dagger$. 
Also observe that \begin{equation}\label{daggerloc}(M_p)^{\dagger p}\simeq  \textstyle  \sum^{-\dim R\slash p}(M^\dagger)_p \end{equation} for all $p\in \Spec R$
(see~\cite[Lemma 1.3.3]{schenzellec}).
Here the dagger dual  on the left-hand side is taken with respect to the normalized dualizing complex of the localization $R_p$. 
  
\begin{Proposition}\label{dib1}Let $(R,m)$ be a local ring admitting a dualizing complex, and let $M\in D^f_b(R)$. Then the following statements are equivalent:
\begin{itemize}
\item[a)]$M^{\dagger}$ is equidimensional;
\item[b)]$\sup M=\sup M_p$   for any $p\in\Supp_RM$.
\end{itemize}
\end{Proposition}
\begin{proof}
 We know by Proposition~\ref{2222} and formula~$(\ref{sdagger})$ that $M^{\dagger}$ being equidimensional is equivalent to $$\dim_{R_p}(M^{\dagger})_p=-\dim R\slash p+\sup M$$ for every $p\in\Supp_RM$. On the other hand, by using formula~$(\ref{daggerloc})$ we get
\begin{align*}
\dim_{R_p}(M^{\dagger})_p&=-\dim R\slash p+\dim_{R_p}(M_p)^{\dagger p}\\&
=-\dim R\slash p+\sup M_p,
\end{align*}
where the last equality comes from formula~$(\ref{sdagger})$. So a) and b) are equivalent.
\end{proof}

\section{ Modules of deficiency of a complex}

We extend the definition given by P.~Schenzel in~\cite[p.~62]{schenzellec} to the case of complexes:

\begin{Definition}\label{kanonicalforcomplex}
Let $(R,m)$ be a local ring admitting a dualizing complex, and let $M\in D^f_b(R)$. For every $i\in \mathbb Z$, set $K^i_M=\Ho_i(M^\dagger)$. The modules $K^i_M$ are called the modules of deficiency of the 
complex $M$. Moreover, we set $K_M=K^{\dim_RM}_M$, and say that $K_M$ is the canonical module of $M$.
\end{Definition}

\begin{Remark}Obviously, the modules of deficiency are finitely generated. Using formulas $(\ref{sdagger})$ and $(\ref{idagger})$, we get
$$ \depth_R M=\inf\{i\in \mathbb Z \mid K^i_M\not=0\}$$
and
$$\dim_R M=\sup\{i\in \mathbb Z \mid K^i_M\not=0\}.$$
Note also that by local duality $\Ho^i_m(M)\cong \Hom_R(K_M^i,E_R(k))$
for all $i\in \mathbb Z$.
\end{Remark}

\begin{Example}\label{S_1} Any finitely generated module is a canonical module of a complex. Indeed, if $K$ is a finitely generated $R$-module and $t\in\mathbb{Z}$, set $M=\textstyle\sum^{-t}K^{\dagger}$. 
Since $\dim_RM=t$ by formula~$(\ref{sdagger})$, it now follows by biduality that $$K_M=\Ho_t(\textstyle\sum^t K)= K.$$
\end{Example}

\begin{Lemma}\label{changcanonical}
Let $(R,m)$ be a local ring admitting a dualizing complex, and let $M\in D^f_b(R)$. Then 
\begin{itemize}
\item[a)] $(K^i_M)_p\cong K^{i-\dim R\slash p}_{M_p}$ for every $p\in \Supp_RM$;
\item[b)]If $p\in\Supp_RM$ with $\dim_RM=\dim_{R_p}M_p+\dim R\slash p$, then $K_{M_p}\cong(K_M)_p$.
\end{itemize}
\end{Lemma}
\begin{proof}

\smallskip
\noindent
a) Using formula $(\ref{daggerloc})$ we get $$
(K^i_M)_p\cong\Ho_i((M^\dagger)_p)\cong \Ho_{i-\dim R\slash p}(M_p)^{\dagger p}=K^{i-\dim R\slash p}_{M_p}. $$
 
\smallskip
\noindent
 
b) Part a) immediately implies that $$K_{M_p}\cong(K_M^{\dim_{R_p} M_p+\dim R\slash p})_p=(K_M)_p.$$
\end{proof}

Our next aim is to investigate the associated primes of modules of deficiency. From now on we set 
\begin{equation*} (X)_i=\bigl\{p\in X\mid\dim R\slash p=i\bigr\}\end{equation*}
for every $X\subset\Spec R$ and  all $i\in\mathbb{Z}$.

\begin{Lemma}\label{assK_M}Let $(R,m)$ be a local ring  admitting a dualizing complex, and let $M\in D^f_b(R)$. Then the following statements hold for all $i\in\mathbb{Z}$:
\begin{itemize}
\item[a)] $\dim_R K^i_M\leq i+\sup M$; 
\item[b)] $(\Ass_RK^{i-s}_M)_i=(\Ass_R \Ho_{s}(M))_i$ where $s=\sup M$; 
\item[c)] $(\Ass_RK_M)_i=(\Ass_R\Ho_{i-\dim_RM}(M))_i$.
\end{itemize}
\end{Lemma}
\begin{proof}

\smallskip
\noindent
a) Using formula $(\ref{dim})$ we have
\[
\dim_RM^\dagger=\sup \bigl\{\dim_RK^i_M-i\mid i\in\mathbb{Z}\bigr\}.
\]
Therefore
\(
\dim_R K^i_M\leq i+\dim_R M^\dagger.
\)
This implies the claim, because $\dim _RM^\dagger=\sup M$  by formula $(\ref{sdagger})$.  

\smallskip
\noindent b) By a) we have $\dim_R K^{i-s}_M\leq i$. Hence, $$(\Ass_RK^{i-s}_M)_i=(\Supp_RK^{i-s}_M)_i.$$ It is then enough to prove that
\[
(\Supp_RK^{i-s}_M)_i=(\Ass_R\Ho_s(M))_i.
\]
Take first $p\in(\Supp_RK^{i-s}_M)_i$. Then $K^{-s}_{M_p}\neq0$ by Lemma~\ref{changcanonical}  a). Therefore $\Ho_s({\bf R}\Gamma_{pR_p}(M_p))\neq0$ implying that 
$$s\leq\sup{\bf R}\Gamma_{pR_p}(M_p).$$ On the other hand, we have
\[
 \sup{\bf R}\Gamma_{pR_p}(M_p)=-\depth_{R_p}M_p
\]
 by formula $(\ref{lsdagger})$. It now follows that 
\[
s\leq\sup{\bf R}\Gamma_{pR_p}(M_p)=-\depth_{R_p}M_p\leq\sup M_p\leq s.
\]
Therefore, $-\depth_{R_p}M_p=s = \sup M_p$. By formula $(\ref{haft})$ this means that $p\in\Ass_R\Ho_s(M)$. So
\[
(\Supp_RK^{i-s}_M)_i\subset(\Ass_R\Ho_s(M))_i.
\]

Conversely, let $p\in(\Ass_R\Ho_s(M))_i$. Then $\depth_{R_p}M_p=-s$ by formula $(\ref{haft})$. Hence $-\sup{\bf R}\Gamma_{pR_p}(M_p)=-s$ implying that $K^{-s}_{M_p}\neq0$. By Lemma~\ref{changcanonical} a) 
this means that
$(K^{i-s}_M)_p\neq0$. Thus $p\in\Supp_RK^{i-s}_M$. Therefore 
 \[
 (\Ass_R\Ho_s(M))_i\subset(\Supp_RK^{i-s}_M)_i.
 \]
\smallskip
\noindent c) This follows by applying b) to $M^\dagger$, because $K_{M^\dagger} \cong \Ho_s(M)$ by formula $(\ref{sdagger})$.

\end{proof}

We can now identify the set of the associated primes and the support of the canonical module of a complex.

\begin{Proposition}\label{suppK_M}Let $(R,m)$ be a local ring  admitting a dualizing complex, and let $M\in D^f_b(R)$. Then 
\begin{itemize}
\item[a)] 
$$\Ass_RK_M=\{p\in\Supp_RM\mid \dim R\slash p=\dim_RM + \inf M_p\};$$
\item[b)] 
$$\Supp_RK_M=\{p\in\Supp_RM\mid\dim_RM=\dim_{R_p}M_p+\dim R\slash p\}.$$
\end{itemize}

\end{Proposition}

\begin{proof}

\smallskip
\noindent a) Let $p\in\Supp_RM$. We apply formula $(\ref{haft})$ to $M^{\dagger}$. Because $\sup M^\dagger=\dim_RM$ by formula $(\ref{sdagger})$, it thus follows that $p\in\Ass_R K_M$ if and only if $\depth_{R_p}(M^\dagger)_p=-\dim_RM$. Using  formulas $(\ref{daggerloc})$ and $(\ref{sdagger})$ we get
\begin{align*}
\depth_{R_p}(M^\dagger)_p&=\depth_{R_p}(M_p)^{\dagger p}-\dim R\slash p\\&=\inf M_p-\dim R\slash p.
\end{align*}
Hence $p\in\Ass_RK_M$ if and only if 
\[
\inf M_p-\dim R\slash p=-\dim_RM.
\]
This proves the claim.
\smallskip

\noindent b)
Let $p\in\Supp_R K_M$. Note first that $\Supp_RM^{\dagger}=\Supp_RM$. Indeed, $\Supp_RM^{\dagger}\subset\Supp_RM$, which implies by biduality that $\Supp_RM\subset
\Supp_RM^{\dagger}$. Since $\Supp_R K_M\subset\Supp_R M^\dagger$, we then have $p\in\Supp_RM$. Take $q\in\Ass_R K_M$ such that $q\subset p$. 
Now
$$\dim_{R_p} M_p \geq \height p\slash q-\inf M_q  =\dim R\slash q-\dim R\slash p-\inf M_q,$$ 
where the first inequality is clear by the definition of Krull dimension and the subsequent equality holds true, since $R\slash q$ is a catenary integral domain.
Because $$\dim R\slash q=\dim_RM +\inf M_q$$ by a), this shows that   
$$\dim_RM \le \dim R\slash p+\dim_{R_p}M_p.$$
Taking into account inequality $(\ref{dimineq})$, we get 
$$\dim_RM=\dim R\slash p+\dim_{R_p}M_p.$$
Suppose then that the above equality holds for $p\in\Supp_RM$. By Lemma~\ref{changcanonical}
b) we have $(K_M)_p\cong K_{M_p}\neq0$. Thus $p\in\Supp_RK_M$,  and we are done.

\end{proof}

\begin{Remark} We observe that by Proposition~\ref{suppK_M} a) and Lemma~\ref{assK_M} c) 
\begin{align*}
\{p\in\Supp_RM\mid \dim R\slash p=\dim_RM + &\inf M_p\}\cr
&=\bigcup_{i\in\mathbb{Z}}(\Ass_R\Ho_{i}(M))_{i+\dim_R M}.
\end{align*} 
\end{Remark}

\begin{Corollary}\label{w_0}Let $(R,m)$ be a local ring  admitting a dualizing complex, and let $M\in D^f_b(R)$. Set $t=\dim_RM$ and $s=\sup M$. Then 
$\dim_RK_M=s+t$ if and only if $\dim_R\Ho_s(M)=s+t$. 
\end{Corollary}

\begin{proof} By Lemma~\ref{assK_M} a) $\dim_RK_M\leq s+t$. Since $\Ho_s(M) \cong K_{M^{\dagger}}$, we also hava  $\dim_R\Ho_s(M)\leq s+t$. Because 
$$(\Ass_RK_M)_{s+t}=(\Ass_R\Ho_s(M))_{s+t}$$ by Lemma~\ref{assK_M} b),  the claim follows.\\
\end{proof}

\section{Main results}

We want to generalize Serre's conditions to complexes. It is convenient to begin with the following very general definition:

\begin{Definition}\label{tildeS_n}
 Let $R$ be a ring. Let $k\in\mathbb{Z}$ and $N\in D^f_b(R)$. We say that a complex $M\in D^f_b(R)$ satisfies Serre's condition $(S_{k,N})$, if  $$\depth_{R_p}{\bf R}\Hom_{R_p}(N_p,M_p)
\geq\min\left\{k,\dim_{R_p}M_p+\inf N_p \right\} $$ for all $p\in\Spec R$.
\end{Definition}

\begin{Remark}Since $$\depth_{R_p}{\bf R}\Hom_{R_p}(N_p,M_p)= \inf N_p+\depth_{R_p} M_p$$ for all $p\in\Spec R$ by~\cite[Proposition 4.6]{iyengar} it is clear that
$(S_{k,N})$ is equivalent to having $$\depth_{R_p}M_p \geq\min\left\{k-\inf N_p,\dim_{R_p}M_p \right\} $$ for all $p\in\Spec R$.
\end{Remark}

The following proposition is now immediate:

\begin{Proposition}\label{SerreCM} Let $R$ be a ring. Let $k\in\mathbb{Z}$ and $N\in D^f_b(R)$. Then a complex $M\in D^f_b(R)$ satisfies the condition $(S_{k,N})$, if 
and only if $M_p$ is Cohen-Macaulay for every $p\in\Spec R$ with   $\depth_{R_p}M_p < k-\inf N_p$.
\end{Proposition}

\begin{Notation}If $N=M$ or $N=\sum^{l}R$ for some $l\in \mathbb Z$, we will speak about the condition $(S_k)$ or $(S_{k,l})$, respectively. In other
words, the complex $M$ satisfies $(S_k)$ if and only if $$\depth_{R_p}M_p \geq\min\left\{k-\inf M_p,\dim_{R_p}M_p \right\} $$ for all $p\in\Spec R$.
Similarly, $M$ is said to satisfy $(S_{k,l})$ if $$\depth_{R_p}M_p \geq\min\left\{k-l,\dim_{R_p}M_p \right\} $$ for all $p\in\Spec R$. 
\end{Notation}

\begin{Remark}\label{supremark}Let $R$ be a ring, and let $M\in D^f_b(R)$. Since $\dim_{R_p}M_p\ge \inf M_p$ for all $p\in \Supp M$, the condition $(S_k)$ holds trivially for $k\le 0$. 
Moreover, as $\inf M_p \le \sup M_p \le \sup M$ , we see that $(S_k)$ always implies  $(S_{k,\sup M})$. 
\end{Remark}

\begin{Remark}
Note that our $(S_k)$ differs from the condition given by Celikbas and Piepmeyer in~\cite[2.4]{Sn} according to which a complex $M\in D^f_b(R)$ satisfies $(S_k)$ if 
$$\depth_{R_p}M_p+\inf M_p\geq\min\left\{k,\height p\right\}$$ for every $p\in\Supp_RM$. However, because $$\dim_{R_p} M_p \le \height p -\inf M_p$$ 
by~\cite[Observation 6.3.3]{918}, we observe that the condition of Celikbas and Piepmeyer implies our $(S_k)$.
\end{Remark}

Given an integer $n$, recall that the \textit{soft truncation} of a complex $M$ above at $n$ is the complex 
\begin{equation*}M_{\subset n}\colon \ldots\rightarrow0\rightarrow\Coker d_{n+1}\rightarrow M_{n-1}\rightarrow M_{n-2}\rightarrow\ldots\qquad .\end{equation*}  
In $D(R)$ we now have an exact triangle
\begin{equation}
\label{tr}\textstyle\sum^n\Ho_n(M) \rightarrow M_{\subset n} \rightarrow M_{\subset n-1}\rightarrow \textstyle\sum^{n+1} \Ho_n(M).
\end{equation} Note that if $n\ge \sup M$, then the natural morphism $M\rightarrow M_{\subset n}$ becomes an isomorphism in $D(R)$.

The next definition was given for modules by P.~Schenzel in~\cite[Definition 4.1]{Schbi}.

\begin{Definition}\label{spectral}
Let $(R,m)$ be a local ring admitting a dualizing complex, and let $M\in D^f_b(R)$. We call the complex $(M^{\dagger})_{\subset\dim_RM-1}$ as the complex 
of deficiency of $M$, and denote it by $C_M$. 
\end{Definition}

\begin{Remark}\label{infCM_p}Clearly \begin{equation}
\Ho_i(C_M)\cong \left\{
  \begin{array}{l l}
   K^i_M,  & \quad \text{if $i\leq \dim_RM-1$;}\\
   0, & \quad \text{otherwise.}\\
  \end{array} \right. 
\end{equation}
In particular, when $M\not\simeq 0$, we have $C_M\simeq 0$ if and only if $M$ is Cohen-Macaulay. If this is not the case,
then $\inf C_M=\depth M$.
\end{Remark}

\begin{Remark}Because $\sup M^\dagger =\dim M$ by formula $(\ref{sdagger})$, we obtain an exact triangle
 \begin{equation}\label{rt}
  \textstyle\sum^{\dim_R M} K_M\rightarrow M^\dagger\rightarrow C_M\rightarrow \textstyle\sum^{\dim_R M+1}K_M.
 \end{equation}
An application of the functor $(-)^\dagger$ to $(\ref{rt})$ yields an exact triangle
  \begin{equation}\label{trr}
   C_M^{\dagger}\rightarrow M\rightarrow\textstyle\sum^{-\dim_RM }K_M^{\dagger}\rightarrow \textstyle\sum^{1}C_M^{\dagger}.
  \end{equation} 
  \end{Remark}

\begin{Lemma}~\label{C{M_p}}Let $(R,m)$ be a local ring admitting a dualizing complex, and let $M\in D^f_b(R)$. If $M$ is equidimensional, then $(C_M^{\dagger})_p\simeq C_{M_p}^{\dagger p}$ for every $p\in\Spec R$.
\end{Lemma}
\begin{proof} Set $t=\dim_RM$. We observe first that $(C_M)_p \simeq\textstyle\sum^{\dim R\slash p}C_{M_p}$. Indeed, by using $(\ref{daggerloc})$ and Proposition~\ref{2222}, we get
\begin{align*}
((M^\dagger)_{\subset t-1})_p
&\cong((M^\dagger)_p)_{\subset t-1}\\
&\simeq(\textstyle\sum^{\dim R\slash p}M_p^{\dagger p})_{\subset t-1}\\
&=\textstyle\sum^{\dim R\slash p}((M_p^{\dagger p})_{\subset t-\dim R\slash p-1})\\
&=\textstyle\sum^{\dim R\slash p}((M_p^{\dagger p})_{\subset \dim_{R_p}M_p-1}).
\end{align*}
Then
 $$(C_M)_p^{\dagger p}\simeq \textstyle\sum^{-\dim R\slash p}C_{M_p}^{\dagger p}$$ so that by $(\ref{daggerloc})$ $(C_M^{\dagger})_p\simeq C_{M_p}^{\dagger p}$ as wanted.
\end{proof}

\begin{Lemma}~\label{mainlemma}Let $(R,m)$ be a local ring admitting a dualizing complex, and let $M, N\in D^f_b(R)$. Suppose that $M$ is an equidimensional complex. Then the following conditions are equivalent:

\begin{itemize}
\item[a)] $M$ satisfies condition $(S_{k,N})$;
\item[b)] $\inf{C_{M_p}}\geq k-\inf N_p$ for all $p\in\Spec R$;
\item[c)] $\sup{\bf R}\Hom_R(N, C_M^{\dagger})\leq -k$;
\item[d)] $\dim_R \Ho_i(N)\otimes_R K^j_M \le i+j-k$ for all $i,j\in \mathbb Z$, $j<\dim_R M$.

\end{itemize}
\end{Lemma}
 \begin{proof}
\smallskip
\noindent
$a)\Leftrightarrow b):$ Let $p\in \Spec R$. We want to show that the conditions $$\depth_{R_p}M_p \geq\min\left\{k-\inf N_p,\dim_{R_p}M_p \right\}$$ and
$\inf{C_{M_p}}\geq k-\inf N_p$ are equivalent. If $M_p$ is Cohen-Macaulay, this is clear. Suppose thus that $M_p$ is not Cohen-Macaulay. We may assume that $M_p\not\simeq 0$. 
By Remark~\ref{infCM_p} the latter condition now means that $\depth_{R_p}M_p \ge k-\inf N_p$. But because $\depth_{R_p}M_p< \dim_{R_p}M_p$, this implies the 
desired equivalence.

\smallskip
\noindent

$b)\Leftrightarrow c):$ It is enough to observe that by~\cite[Lemma 7.2.7]{918}, Lemma~\ref{C{M_p}} and formula $(\ref{idagger})$ we have
\begin{align*}
-\sup{\bf R}\Hom_R(N,C_M^{\dagger})
&=\inf\left\{\depth_{R_p}(C_M^{\dagger})_p+\inf N_p\mid p\in\Spec R\right\} \\
&=\inf\left\{\depth_{R_p} C_{M_p}^{\dagger p}+\inf N_p\mid p\in\Spec R\right\} \\
&=\inf\left\{\inf C_{M_p}+\inf N_p\mid p\in\Spec R\right\}. 
\end{align*}

\smallskip
\noindent

$c)\Leftrightarrow d):$ Using adjointness and formula $(\ref{sdagger})$ we get
$$\sup{\bf R}\Hom_R(N,C_M^{\dagger})=\sup (N \otimes_R^L C_M)^{\dagger}=\dim_R N\otimes_R^L C_M.$$
The claim follows, because by~\cite[Proposition 6.3.9 b), (E.6.3.1)]{918} and Remark~\ref{infCM_p} we have
\begin{align*}
\dim_R N &\otimes_R^L C_M\cr
&=\sup\{\dim_R \Ho_i(N)\otimes_R^L C_M-i \mid i \in \mathbb Z\}\\
&=\sup\{\sup \{\dim_R \Ho_i(N)\otimes_R^L \Ho_j(C_M)-j\mid j\in \mathbb Z\}-i \mid i \in \mathbb Z\}\\
&=\sup\{\dim_R \Ho_i(N)\otimes_R^L \Ho_j(C_M)-i-j \mid i,j\in \mathbb Z\}\\
&=\sup\{\dim_R \Ho_i(N)\otimes \Ho_j(C_M)-i-j \mid i,j\in \mathbb Z\}\\
&=\sup\{\dim_R \Ho_i(N)\otimes K^j_M-i-j \mid i,j\in \mathbb Z\,j<\dim_R M\}.
\end{align*}

\end{proof}

\begin{Theorem}\label{maintheorem}Let $(R,m)$ be a local ring admitting a dualizing complex. Let $k\in \mathbb Z$ and $M, N\in D^f_b(R)$. Set $t=\dim_RM$. If $M$ is equidimensional, then the following 
conditions are equivalent:
\begin{itemize}
\item[a)]$M$ satisfies condition $(S_{k,N})$;
\item[b)] The natural homomorphism $\Ext^{-i}_R(N,M)\rightarrow K^{i+t}_{N\otimes^L_RK_M}$ is bijective for all $i\geq -k+2$, and injective for $i=-k+1$.
\end{itemize}
\end{Theorem}

\begin{proof} 
By applying the functor ${\bf R}\Hom_R(N,-)$ on~$(\ref{trr})$, we get the exact triangle
\begin{align*}
 {\bf R}\Hom_R(N,C_M^{\dagger})\rightarrow{\bf R}\Hom_R(N, M){\rightarrow}\textstyle\sum^{-t}&{\bf R}\Hom_R(N,K_M^{\dagger})\\
&\rightarrow \textstyle\sum^{1}{\bf R}\Hom_R(N,C_M^{\dagger}).
\end{align*}
Observe that ${\bf R}\Hom_R(N,K_M^{\dagger})\simeq (N\otimes^L_RK_M)^\dagger$ by adjointness. Since $(S_{k,N})$ is by Lemma~\ref{mainlemma} equivalent to $\sup{\bf R}\Hom_R(N,C_M^{\dagger})\le -k$, a look at the corresponding long exact sequence of homology implies the claim.
\end{proof}

In particular, by taking $N=M$, this immediately applies to Serre's condition $(S_k)$. For this case, we observe the following

\begin{Proposition}\label{tensorcanonical}Let $(R,m)$ be a local ring admitting a dualizing complex. Let $M\in D^f_b(R)$. Set $t=\dim M$. Then $\dim_R M \otimes^L_R K_M=t$.  Moreover, the natural homomorphism
$$\Ext^{-i}_R(K_M,K_M)\rightarrow K^{i+t}_{M\otimes^L_RK_M}$$ is an isomorphism for $i>\sup C_M-t$. In particular, $K_{M \otimes^L_R K_M}=\Hom_R(K_M,K_M)$. 
\end{Proposition}

\begin{proof}
By adjointness $$(M\otimes^L_R K_M)^{\dagger} \simeq {\bf R}\Hom_R(K_M,M^{\dagger}).$$ Since $\Hom_R(K_M,K_M)\not=0$, it follows from formula $(\ref{sdagger})$ and~\cite[Proposition A.4.6]{C.L.N} 
that $$\dim M \otimes^L_R K_M= \sup {\bf R}\Hom_R(K_M,M^{\dagger})=\sup M^{\dagger}-\inf K_M=\dim M.$$ An application of the functor ${\bf R}\Hom_R(K_M,-)$ on~$(\ref{rt})$, yields the exact triangle
\begin{align*}
 \textstyle\sum^{-t-1}{\bf R}\Hom_R(K_M,C_M) \rightarrow {\bf R}&\Hom_R(K_M,K_M)\\ &\rightarrow \textstyle\sum^{-t}{\bf R}\Hom_R(K_M,M^\dagger)\\
&\hskip1.5truecm\rightarrow \textstyle\sum^{-t}{\bf R}\Hom_R(K_M,C_M).
\end{align*}
The desired isomorphism now follows from the corresponding long exact sequence of homology, because $\sup {\bf R}\Hom_R(K_M,C_M) \le \sup C_M$ by~\cite[Proposition A.4.6]{C.L.N}.

\end{proof}

Let us then consider the conditions $(S_{k,l})$.

\begin{Corollary}
\label{schenzeltheorem}Let $(R,m)$ be a local ring admitting a dualizing complex. Let $k,l\in \mathbb Z$ and $M\in D^f_b(R)$. Set $t=\dim_RM$. If $M$ is equidimensional, then the following conditions are equivalent:
\begin {itemize}
\item[a)] $M$ satisfies  $(S_{k,l})$;
\item[b)]The natural homomorphism $\Ho_i(M) \rightarrow K^{i+t}_{K_M}$ is bijective for $i\geq l-k+2$, and injective for $i=l-k+1$;
\item[c)]The natural homomorphism $$\Ho_m^{i+t}(K_M)\rightarrow \Hom_R(\Ho_i(M),E_R(k))$$ is bijective for $i\geq l-k+2$, and surjective for $i=l-k+1$.
\end{itemize}

\end{Corollary}

\begin{proof} The equivalence of a) and b) follows immediately from Theorem~\ref{maintheorem} by taking $N=\sum^{l}R$ whereas the homomorphism of c) is by local duality the Matlis-dual of that of b).
\end{proof}

If $R$ is a ring and $N$ is an $R$-module, we use the notation 
$$\Assh_R N=\{p\in \Supp_R N\mid \dim R\slash p=\dim_R N\}.$$

\begin{Corollary}\label{H_SK_M}Let $(R,m)$ be a local ring admitting a dualizing complex. Let $M\in D^f_b(R)$ be an equidimensional complex. Set $t=\dim_RM$ and $s=\sup M$. If $M$ satisfies 
Serre's condition $(S_1)$, then 
\begin{itemize}
\item[a)]$\dim_R\Ho_s(M)=\dim_RK_M=s+t$;
\item[b)]$\Ass_R\Ho_s(M)=\Assh_R\Ho_s(M)$.
\end{itemize}
\end{Corollary}
\begin{proof} $a)$ Recall that $(S_1)$ implies $S_{1,s}$. By Corollary~\ref{schenzeltheorem} the natural homomorphism $\Ho_s(M)\rightarrow K^{s+t}_{K_M}$ is injective. Then $K^{s+t}_{K_M}\neq0$ so that by Lemma~\ref{assK_M} a) we must have
$\dim_RK_M=s+t$. It now follows from Corollary~\ref{w_0} that $\dim\Ho_s(M)=s+t$, too.

\smallskip
\noindent$b)$ Because $\dim_R\Ho_s(M)=s+t$ by a), it is enough to show that $\Ass_R\Ho_s(M)=(\Ass_R\Ho_s(M))_{s+t}$. By a) we also have an injective homomorphism $\Ho_s(M)\rightarrow K_{K_M}$. So $$\Ass_R\Ho_s(M)\subset \Ass_RK_{K_M}=(\Ass_RK_M)_{s+t},$$ where the last equality comes from~\cite[Proposition 2.3 b)]{Schbi}. Since $(\Ass_RK_M)_{s+t}=(\Ass_R\Ho_s(M))_{s+t}$
by Lemma~\ref{assK_M} b),  we get $$\Ass_R\Ho_s(M)=(\Ass_R\Ho_s(M))_{s+t}$$ as wanted.
\end{proof}

\begin{Corollary}\label{H_SK_M2'}Let $(R,m)$ be a local ring admitting a dualizing complex. Let $M\in D^f_b(R)$ be an equidimensional complex. If $M$ satisfies Serre's condition $(S_2)$, then 
$$\Hom_{D(R)}(M,M)\cong \Hom_R(K_M,K_M)$$ and $\Ho_s(M)\cong K_{K_M}$. Moreover, if $K_M$ is equidimensional and satisfies Serre's condition $(S_2)$, then $K_M\cong K_{\Ho_s(M)}$.
\end{Corollary}

\begin{proof}Recall first that $\Hom_{D(R)}(M,M)\cong \Ext_R^0(M,M)$. The desired isomorphism $\Hom_{D(R)}(M,M))\cong \Hom_R(K_M,K_M)$ then comes from Theorem~\ref{maintheorem} by taking $N=M$ and Proposition~\ref{tensorcanonical} whereas Corollary~\ref{schenzeltheorem} and Corollary~\ref{H_SK_M} a) provide the isomorphism $\Ho_s(M)\rightarrow K_{K_M}$. Combining the latter with~\cite[Theorem 1.14 (ii)]{schenzel}, shows that $K_{\Ho_s(M)}\cong K_{K_{K_M}}\cong K_M$.   
\end{proof}

We now turn to look at the dagger dual:

\begin{Proposition}\label{dagserS_n}Let $(R,m)$ be a local ring admitting a dualizing complex, and let $M\in D^f_b(R)$. The complex $M^{\dagger}$ satisfies Serre's condition $(S_k)$ if and only if $\sup M_p=
\inf M_p$ for every $p\in\Supp_RM$ with $\depth M_p + \inf M_p < k$.
\end{Proposition}
\begin{proof} Note first that by using formula~$(\ref{daggerloc})$ together with formulas $(\ref{sdagger})$ and $(\ref{idagger})$ we get
$$\dim_{R_p}(M^{\dagger})_p=-\dim R\slash p+\sup M_p$$
and
$$\depth_{R_p}(M^{\dagger})_p=-\dim R\slash p+\inf M_p.$$
We also have
$$\inf (M^{\dagger})_p = \dim R\slash p+\inf M_p^{\dagger p}=\dim R\slash p+\depth_{R_p} M_p.$$
The claim then follows from Proposition~\ref{SerreCM}. 
\end{proof}

In order to apply Theorem~\ref{maintheorem} in this case, we need

\begin{Lemma}\label{Extlemma}Let $(R,m)$ be a local ring admitting a dualizing complex, and let $M\in D^f_b(R)$. Then $\Ext_R^i(M^\dagger,M^\dagger)\cong\Ext_R^i(M,M)$ for all $i\in \mathbb Z$.
\end{Lemma}

\begin{proof}By `swap' (see~\cite[A.4.22]{C.L.N}) and dagger duality
$${\bf R}\Hom_R(M^\dagger,M^\dagger) \simeq {\bf R}\Hom_R(M, M^{\dagger\dagger}) \simeq {\bf R}\Hom_R(M,M).$$
The claim now follows by taking the homology.
\end{proof}

\begin{Corollary}\label{endotheorem2}Let $(R,m)$ be a local ring admitting a dualizing complex. Let $k\in \mathbb Z$ and $M\in D^f_b(R)$. Set $s=\sup M$. If $M^\dagger$ is equidimensional, then the following statements are equivalent:
\begin{itemize}
\item[a)]$M^\dagger$ satisfies condition $(S_k)$;
\item[b)] The natural homomorphism $\Ext^{-i}_R(M,M)\rightarrow K^{i+s}_{{\bf R}\Hom_R(\Ho_s(M), M)}$ is bijective for all  $i\geq -k+2$, and injective for $i=-k+1$.
\end{itemize}
\end{Corollary}

\begin{proof} Note that $\dim_R M^\dagger = s$ by formula~$(\ref{sdagger})$. By dagger duality $K_{M^\dagger}=\Ho_s(M)$. By adjointness and biduality
we then get
$$(M^\dagger \otimes_R^L K_{M^\dagger})^\dagger \simeq {\bf R}\Hom_R(K_{M^\dagger},M^{\dagger\dagger}) \simeq {\bf R}\Hom_R(\Ho_s(M), M).$$
The claim is then a direct consequence of
Lemma~\ref{Extlemma} and Theorem~\ref{maintheorem}.
\end{proof}

In a similar way, Corollary~\ref{schenzeltheorem} yields

\begin{Corollary}
\label{schenzeltheorem2}Let $(R,m)$ be a local ring admitting a dualizing complex, $k,l\in \mathbb Z$ and let $M\in D^f_b(R)$. Set $s=\sup M$. If $M^\dagger$ is equidimensional, then the following 
conditions are equivalent:
\begin {itemize}
\item[a)] $M^{\dagger}$ satisfies Serre's condition $(S_{k,l})$;
\item[b)]The natural homomorphism $K^i_M\rightarrow K^{i+s}_{\Ho_s(M)}$ is bijective for $i\geq l-k+2$, and injective for $i=l-k+1$;
\item[c)]The natural homomorphism $\Ho_m^{i+s}(\Ho_s(M))\rightarrow\Ho^i_m(M)$ is bijective for $i\geq l-k+2$, and surjective for $i=l-k+1$.
\end{itemize}
\end{Corollary}

Corollary~\ref{H_SK_M} and Corollary~\ref{H_SK_M2'} have now the following analogues:

\begin{Corollary}\label{H_SK_M2}Let $(R,m)$ be a local ring admitting a dualizing complex and let $M\in D^f_b(R)$.  Suppose that $M^\dagger$ is equidimensional. Set $t=\dim_RM$ and $s=\sup M$. 
If $M^\dagger$ satisfies Serre's condition $(S_1)$, then 
\begin{itemize}
\item[a)]$\dim_R\Ho_s(M)=\dim_RK_M=s+t$;
\item[b)]$\Ass_RK_M=\Assh_RK_M$.
\end{itemize}
\end{Corollary}

\begin{Corollary}\label{H_SK_M2''}Let $(R,m)$ be a local ring admitting a dualizing complex, and let $M\in D^f_b(R)$.  Suppose that $M^\dagger$ is equidimensional. Set $s=\sup M$. If $M^\dagger$ satisfies 
Serre's condition $(S_2)$, then $$\Hom_{D(R)}(M,M) \cong \Hom_R(\Ho_s(M),\Ho_s(M))$$ and $K_M\cong K_{\Ho_s(M)}$. Moreover, if $\Ho_s(M)$ is equidimensional and satisfies Serre's condition $(S_2)$, then $\Ho_s(M)\cong K_{K_M}$. 
\end{Corollary}

Let $R$ be a ring. Recall that a filtration of $\Spec R$ is a descending sequence $$\mathcal{F}\colon\quad \ldots \supseteq F^{i-1}\supseteq F^i\supseteq F^{i+1} \supseteq\ldots $$ of subsets of $\Spec R$ such that $\bigcap_i F^i=\emptyset$, $F^i=\Spec R$ for some $i\in \mathbb Z$ and each $p\in F^i \setminus  F^{i+1}$ is a minimal element of $F^i$ with respect to inclusion. Let $E_\mathcal{F}(M)$ denote the Cousin complex corresponding to a complex $M\in D^f_b(R)$. Recall that $E_\mathcal{F}(M)$ is a complex $$\ldots \rightarrow E_\mathcal{F}(M)_i \rightarrow E_\mathcal{F}(M)_{i-1}\rightarrow\ldots $$ with $$E_\mathcal{F}(M)_i=\bigoplus_{p\in F^{-i}\setminus F^{-i+1}} \Ho^{-i}_{pR_p}(M_p).$$ Observe that we here grade the Cousin complex homologically in contrary to the general tradition. For more details about Cousin complexes
we refer to~\cite[Chapter IV, \S 3]{lnsfunctorial}. Note that if $M$ is an $R$-module, then the Cousin complex studied by Sharp (see~\cite{sharptang}, for example) is a complex 
$$0 \rightarrow M \rightarrow E_\mathcal{F}(M)_0 \rightarrow  E_\mathcal{F}(M)_{-1} \rightarrow \ldots$$
 A standard example of a filtration is the ``$M$-dimension filtration'' ${\mathcal D}(M)$ defined by the formula  $${\mathcal D}^i(M)=\left\{p\in \Spec R\mid i\leq \dim M-\dim R\slash p\right\}$$ for all $i\in\mathbb{Z}$.

The following result is proved by Lipman,  Nayak and Sastry in~\cite[Proposition 9.3.5]{lnsfunctorial}. Similar results have been proved in the module case by Dibaei and Tousi in~\cite[Theorem 1.4]{tusidib} 
and by Kawasaki in~\cite[Theorem 5.4]{Kawasaki}. 
 \begin{Proposition}\label{lipman2} Let $(R,m)$ be a local ring admitting a dualizing complex, and let $M\in D^{f}_b(R)$. Then 
 $$E_{{\mathcal D}(M)}(M)\cong\textstyle\sum^{-\dim_RM}K_M^\dagger.$$
\end{Proposition}

We can now use this to prove

\begin{Proposition}\label{MK_ME2}Let $(R,m)$ be a local ring admitting a dualizing complex, and let $M\in D^f_b(R)$. 
Set $s=\sup M$. If $M^{\dagger}$ is equidimensional and satisfies Serre's condition $(S_2)$, then $$E_{{\mathcal D}(M)}(M)\cong \textstyle\sum^{s}E_{{\mathcal D}(\Ho_s(M))}(\Ho_s(M)).$$ 
\end{Proposition}
\begin{proof}Set $t=\dim_RM$. Note that $\dim_R\Ho_s(M)=s+t$ by  Corollary~\ref{H_SK_M2}. Since $K_{\Ho_s(M)}\cong K_M$ by Corollary~\ref{H_SK_M2''}, Proposition~\ref{lipman2} gives
$$E_{{\mathcal D}(\Ho_s(M))}(\Ho_s(M))\cong\textstyle\sum^{-s-t} K_{\Ho_s(M)}^\dagger \cong\textstyle\sum^{-s-t} K_{M}^\dagger \cong\textstyle\sum^{-s}E_{{\mathcal D}(M)}(M).$$
\end{proof}

\bibliographystyle{amsplain}

\bibliography{akhavinhyryJPA}

\end{document}